\documentclass{article}
\usepackage{graphics}
\usepackage{graphicx}
\usepackage{latexsym}
\usepackage{amssymb}
\usepackage{amsmath}

\newtheorem{theorem}{Theorem}[section]
\newtheorem{lemma}[theorem]{Lemma}
\newtheorem{proposition}[theorem]{Proposition}
\newtheorem{corollary}[theorem]{Corollary}

\newtheorem{preexample}{Example}[section]
\newenvironment{example}{\begin{preexample}}{\end{preexample}}
\newtheorem{preremark}{Remark}

\newcommand{\qed }{ \hfill $\Box$ }

\newcommand{\g}{\mathfrak{g}}

\title{A topological conjugacy of invariant flows on some class of Lie groups}
\author{Alexandre J. Santana  and Sim\~ao N. Stelmastchuk}

\begin{document}

\maketitle

\begin{abstract}
  The aim of this paper is to give a condition to topological conjugacy of invariant flows in an Lie group $G$ which its Lie algebra $\g$ is associative algebra or semisimple. In fact, we show that if two dynamical system on $G$ are hyperbolic, then they are topological conjugate.  
\end{abstract}



\section{Introduction}

Let $G$ be a Lie group with Lie algebra $\g$. Taking $A \in \g$ we consider the dynamical system 
\begin{equation}\label{systemLiegroup}
  \dot{g} = A_g = R_{g*}(A), \ \ g \in G.
\end{equation}
We are interested in establishing a condition to topological conjugacy to dynamical system (\ref{systemLiegroup}).  We will restrict our attention to two class of the Lie groups, when $\g$ is an associative or semisimple Lie algebra.  Our idea is to study the prime case and after to apply this in the semisimple case. 

In following we explain our idea to work when $\g$ is an associative Lie algebra. We begin by remembering that the topological conjugacy of linear system in $\mathbb{R}^n$ given by
\begin{equation}\label{systemRn}
  \dot{x} = A x, \ \ x \in \mathbb{R}^n, \ \ A \in gl(n,\mathbb{R})
\end{equation}
is well established, see for instance the textbooks due to Robinson \cite{robinson} and due to Colonius and Kliemann \cite{coloniuskliemann}. It is well known that the main tool to this classification is the hyperbolic property. Namely, a dynamical system (\ref{systemRn}) has the hyperbolic property if there exists an $a>0$ such that $\| e^{At}x\| \leq e^{-at}\|x\|$ for all $t>0$. Here we observe that under this study there is a normed space structure. 

Observing the developments to dynamical system (\ref{systemRn}) our idea is to introduce a normed space structure to work with the dynamical system (\ref{systemLiegroup}). Hovewer, it is clear that $G$, in general, does not have a normed space structure. Thus instead of work directly with $G$ we lift the dynamical system (\ref{systemLiegroup}) to one on $G\times \g$, which inherits a normed vector structure from the Lie algebra $\g$. This idea follows from works due to Ayala, Colonius and Kliemann \cite{ayalacolonius}, Colonius and Santana \cite{santana}, \cite{santana2}. Being more clear, we lift the dynamical system (\ref{systemLiegroup}) to one on $G\times \g$ given by
\begin{equation}\label{systemGg}
  \left(
  \begin{array}{c}
  \dot{g}\\
  \dot{v}
  \end{array}
  \right)
  =
    \left(
  \begin{array}{c}
  R_{g*}A\\
  Av
  \end{array}
  \right).
\end{equation}
It is true that the solution of this dynamical system is $(e^{At}g, e^{At}v)$. Thus if we suppose that two dynamical system on $G \times \g$ are hyperbolic, then they are topological conjugate, and, as consequence, we show that $e^{At}g$ and  $e^{Bt}g$ are topologically conjugate in $G$. This is assured by Theorem \ref{topologicallyconjugate}. 

When we consider $G$ as the linear group $GL(n, \mathbb{R})$ and we work with dynamical systems of kind $\dot{X} = A X$ where $X \in GL(n,\mathbb{R})$, we use  Theorem \ref{topologicallyconjugate} to recover the classical result by dynamical system (\ref{systemRn}): if real part of generalized eigenvalues of $A$ and $B$ are negative, then $e^{At}$ and $e^{Bt}$ are topological conjugate. As a simple application we give a partial classification for dynamical system of $GL(2,\mathbb{R})$.

To end we consider the dynamical system (\ref{systemLiegroup}) in a semisimple Lie group. Here our idea is to study the dynamical system (\ref{systemLiegroup}) in case of the Adjoint group $Ad(G)$. In fact, we show that if real part of generalized eigenvalues of $ad(A)$ and $ad(B)$ are negative, then $e^{At}$ and $e^{Bt}$ are topological conjugate. Finally, using previous results we classify dynamical system in an affine group $H\rtimes V$, where $H$ is a semisimple Lie group and $V$ a $n$-dimensional vector space, or rather we state and prove a condition to topological conjugacy of flows in $H\rtimes V$.

\section{Hyperbolic condition to topological conjugate in Lie groups}

Let $G$ be a Lie group and $\g$ its Lie algebra. In this section we assume $\g$ to be an associative algebra. This structure is natural in matrix groups. We adopt in the trivial vector bundle $G \times \g$ the direct product
\[
  (g,v) \cdot (h,u) = (g \cdot h, v \cdot u).
\]
Our main interesting is to study as the hyperbolic property entails topological conjugacy to dynamical systems of kind
\[
  \dot{g} = R_{g*} A , \ \ A \in \g
\]
in the Lie group $G$. Our idea is to lift this study to trivial bundle $G \times \g$ and to study the dynamical system 
\begin{equation}\label{systemGxg}
  \left(
  \begin{array}{c}
  \dot{g}\\
  \dot{v}
  \end{array}
  \right)
  =
    \left(
  \begin{array}{c}
  R_{g*}A\\
  Av
  \end{array}
  \right).
\end{equation}
A direct account shows that $(e^{At}g, e^{At}v)$ is a solution of (\ref{systemGxg}) if $g \in G$ and $v \in \g$.

Since hyperbolic property of dynamical system needs a norm, we begin by adopting a norm $\| \cdot \|_g$ in $\g$. In consequence, in the trivial vector bundle $G \times \g$ we consider the following norm
\[
  \|(g,v)\| = \|v\|_\g.
\]
For this reason,
\[
  \|(e^{At}g, e^{At}v)\| = \|e^{At}v\|_\g.
\]
In consequence, the behavior of the flow $e^{At}v$ drives the behavior of the folw $(e^{At}g, e^{At}v)$ in $G \times \g$.

\begin{lemma}\label{uniquetime}
  Let $A \in \g$. We assume that $e^{At}v$ has the hyperbolic property, that is, there exists $a>0$ such that for every $t\geq 0$ we have
  \begin{equation}\label{hyperpebolic}
    \|  e^{At}v \|_\g \leq e^{-at} \| v\|_\g.
  \end{equation}
  Then there exists a unique time $\tau$ such that $\|(e^{A\tau}g, e^{A\tau}v)\|=1$.
\end{lemma}
\begin{proof}
  We first observe that if $e^{At}v$ has the hyperbolic property, then 
  \[
    \|(e^{At}g, e^{At}v)\| \leq e^{-at} \| (g,v)\|, \ \ t \geq 0.
  \]
  On the other hand, when $t<0$ we have 
  \[
    \| e^{At}v \|_\g \geq e^{a|t|} \| v\|_\g,
  \]
  and, consequently,
  \[
    \|(e^{At}g, e^{At}v)\| \geq e^{a|t|} \| (g,v)\|.
  \]
  It implies that when $t \to \pm \infty$ it follows that $\|(e^{At}g, e^{At}v)\|$ goes to $\pm \infty$, respectively. Since $\|(\cdot, \cdot) \|$ is a continuous  function, there is a time $\tau$ such that
  \[
    \|(e^{A\tau}g, e^{A\tau}v)\| = 1.
  \]
  We claim that this time is unique. In fact, suppose that there exist $\tau_1, \tau_2 \in \mathbb{R}$ such that $\| (e^{A\tau_1}g, e^{A\tau_1}v)\| = \| (e^{A\tau_2}g, e^{A\tau_2}v) \| = 1$. Assume that $\tau_2 \geq \tau_1$. If we write $\tau_2 = (\tau_2 - \tau_1) + \tau_1$, then
  \[
    1= \| e^{A\tau_2}v\|_\g = \| e^{A (\tau_2 - \tau_1)} e^{A \tau_1}v\|_\g.
  \]
  From hypothesis we have
  \[
    \|e^{A (\tau_2 - \tau_1)} e^{A \tau_1}v\| \leq e^{-a(\tau_2 -\tau_1)}\|e^{A \tau_1}v\| = e^{-a(\tau_2 -\tau_2)}.
  \]
  It follows that $ 1 = e^{-a(\tau_2 -\tau_1)}$, and so $-a(\tau_2 -\tau_1) = 0$. Since $a>0$, it entails that $\tau_2 = \tau_1$. \qed
\end{proof}

\begin{theorem}\label{topologicallyconjugate}
  If $A, B \in \g$ such that $e^{At}v$ and $e^{Bt}v$ satisfy the hyperbolic property, then $(e^{At}g, e^{At}v)$ and  $(e^{Bt}g, e^{Bt}v)$ are topologically conjugate
	in $G \times \g$. In consequence, $e^{At}g$ and  $e^{Bt}g$ are topologically conjugate in $G$.
\end{theorem}
\begin{proof}
  We first assume that there exist $a,b>0$ such that for every $t \geq 0$ we have
  \begin{eqnarray*}
    \|(e^{At}g, e^{At}v)\| & \leq & e^{-at} \| (g,v)\|\\
    \|(e^{Bt}g, e^{Bt}v)\| & \leq & e^{-bt} \| (g,v)\|.
  \end{eqnarray*}
  For every $t<0$, it is direct that
  \begin{eqnarray*}
    \|(e^{At}g, e^{At}v)\| & \geq & e^{a|t|} \| (g,v)\|\\
    \|(e^{Bt}g, e^{Bt}v)\| & \geq & e^{b|t|} \| (g,v)\|.
  \end{eqnarray*}
  Denote the unit sphere in $G\times \g$ by $\mathbb{S}$, that is,
  \[
    \mathbb{S} = \{ (g,v) : \|(g,v)\| = 1\}.
  \]
  From Proposition (\ref{uniquetime}) it is clear that the flows $(e^{At}g, e^{At}v)$ and $(e^{Bt}g, e^{Bt}v)$ cross $\mathbb{S}$ in an unique point.

  Define the map $t_{A}: G\times \g  \rightarrow \mathbb{R}$ that associates for every $(g,v) \in G\times\g$ the unique time $t_{A}(g,v)$ such that $\|(e^{At_{A}(g,v)}g,e^{At_{A}(g,v)}v)\| =1$. We claim that $t_{A}$ is continuous. In fact, take $(g,v) \in G\times \g$ and a sequence $(g_n,v_n) \to (g,v)$. We observe that $v_n \rightarrow v$, that is,  $\| v_n - v \|_g \to 0$. For simplicity, suppose that $t_A(g_n,v_n) > 0$ for every $n \in \mathbb{N}$. Thus we have
  \[
    \| e^{A t_A(g_nv_n)}(v_n - v) \|_\g \leq e^{-a t_A(g_n,v_n)}\|v_n - v\|_\g \to 0.
  \]
  This gives that
  \[
    \lim_{n \to \infty}e^{A t_A(g_n,v_n)}(v_n)= \lim_{n \to \infty} e^{A t_A(g_n,v_n)} v.
  \]
  Since $\| \cdot \|_\g$ is a continuous map,
  \[
    \lim_{n \to \infty}\|e^{A t_A(g_n,v_n)}(v_n)\|_\g= \| e^{A \lim_{n \to \infty}t_A(g_n,v_n)} v\|_g
  \]
  As $\|e^{A t_A(g_n,v_n)}(v_n)\|_\g = 1$ we have $\| e^{A \lim_{n \to \infty}t_A(g_n,v_n)} v\|_\g =1$. In according with uniqueness we obtain $\lim_{n \to \infty}t_A(g_n,v_n) = t_A(g,v)$. In consequence, $t_A$ is a continuous map.

  Before we construct our homeomorphism we need to show the following property: $t_A (e^{At}g,e^{At}v) = t_A(g,v) - t$. In fact,
  \[
    1=\|e^{A t_A(g,v)}v\|_\g = \|e^{A t_A(g,v)}e^{A (-t)}e^{A t}v\|_\g = \|e^{A (t_A(g,v)-t)}e^{A t}v\|_\g.
  \]
  By uniqueness, we conclude that $t_A(e^{At} g,e^{At}v) = t_A(g,v) - t$.

  We define the map $\psi_0 : \mathbb{S} \rightarrow \mathbb{S}$ by $\psi_o(g,v) = (g,\frac{v}{\| v  \|_{\mathfrak{g}}})$ and the map $\phi_0 : \mathbb{S} \rightarrow \mathbb{S}$ by	$\phi_o(Y) = (g^{-1},\frac{v}{\|v\|_\g})$. It is clear that $\phi_0 = \psi_0^{-1}$. It is clear that that $\psi_0$ and $\phi_0$ are continuous maps. Hence  $\psi_0$ is an homeomorphism.  
   
  Our next goal is to extend the homeomorphism $\psi_0$ to $G\times \g$. We begin by defining the map $\psi: G\times \g \rightarrow G\times \g$ by
  \[
    \psi(g,v)=
    \left\{
    \begin{array}{cc}
      ( e^{-B\, t_{A}(g,v)}e^{A\, t_{A}(g,v)}g, e^{-B\, t_{A}(g,v)}h_{0}(e^{A t_{A}(g,v)}v)) & \mbox{ if } v \neq 0\\
       (g,0) & \mbox{ if } v = 0 .
    \end{array}
    \right.
  \]
  We write $\psi(g,v) = (\psi_1(g), \psi_2(v))$, where
  \[
    \psi_1(g)= e^{-B\, t_{A}(g,v)}e^{A\, t_{A}(g,v)}g
  \]
  and
  \[
    \psi_2(v)=
    \left\{
    \begin{array}{cc}
      e^{-B\, t_{A}(g,v)}h_{0}(e^{A t_{A}(g,v)}v) & \mbox{ if } v \neq 0\\
      0 & \mbox{ if } v = 0 .
    \end{array}
    \right.
  \]
  We proceed to show the conjugacy. In fact
  \begin{eqnarray*}
    \psi(e^{A t}g, e^{A t}v)
    & = & (e^{-B\, t_{A}(e^{A t} g,e^{A t} v)}e^{A t_{A}(e^{A t} g,e^{A t} v)}e^{A t} g,\,\, e^{-B\, t_{A}(e^{A t} g,e^{A t} v)}h_{0}(e^{A t_{A}(e^{A t} g,e^{A t} v)}e^{A t} v)\\
    & = & (e^{-B\, [t_{A}(g,v) -t]}e^{A [t_{A}(g,v) -t]}e^{A t} g,\,\, e^{-B\, [t_{A}(g,v) -t]}h_{0}(e^{A [t_{A}(g,v) -t]}e^{A t} v))\\
    & = & (e^{B t} e^{-B\, t_{A}(g,v)}e^{A t_{A}(g,v)} g,\,\, e^{B t} e^{-B\, t_{A}(g,v)}h_{0}(e^{A t_{A}(g,v)} v))\\
    & = & (e^{B t} \psi_1(g),\,\, e^{B t} \psi_2(v))\\
    & = & (e^{B t}, e^{B t})(\psi_1(g),\psi_2(v))\\
    & = & (e^{B t}, e^{B t})\psi(g,v).
  \end{eqnarray*}
  The next step is to show that $\psi$ is continuous. We need only consider the case $(g,v)$ with $v=0$. In fact, the map $\psi=(\psi_1,\psi_2)$ is continuous if $v\neq 0$ because $(e^{At}g,e^{At}v)$, $(e^{Bt}g,e^{Bt}v)$, $\tau_A$ and $\tau_B$ are continuous. We begin by observing that $\psi_1$ is continuous, so we need verifying the continuity of the map $\psi_2$. Thus fix a $g \in G$ and take a sequence $v_n$ such that $v_n \to 0$ in $\g$. Our work is to show that
  \[
    \psi_2(v_n) \rightarrow \psi_2(0) = 0.
  \]
  If $t_{A}(g_n,v_n)$ is the time such that $\|e^{A t_{A}(g_n,v_n)}v_n\|_\g =1$, a well known argument shows that $ t_n = t_A(g_n,v_n) \to -\infty$.
  Let us denote $u_n = \psi_0(e^{At_{n}}v_{n})$. Thus
  \[
    \|u_n\|_\g = \|\psi_0(e^{At_{n}}v_{n}) \|_\g = \frac{\|e^{At_{n}}v_{n}\|_\g}{\|e^{At_{n}}v_{n} \|_\g} = 1
  \]
  and, for this,
  \[
    \|\psi_2(v_n)\|_\g = \|e^{-Bt_{n}}\psi_0(e^{At_{n}}v_{n}) \|_\g \leq = \|e^{-Bt_{n}}\| \|\psi_0(e^{At_{n}}v_{n}) \|_\g = \|e^{-Bt_{n}}\| \leq e^{bt_{n}} \to 0
  \]
  We thus conclude that $\psi_2(v_n) \to 0$ and, therefore, $\psi_2$ is continuous at $0 \in \g$. Consequently, $\psi$ is continuous. 

  To end,  we define the map
  \[
    \psi^{-1}(g,v)=
    \left\{
    \begin{array}{cc}
      ( e^{-A\, t_{B}(g,v)}e^{B\, t_{B}(g,v)}g, e^{-A\, t_{B}(g,v)}h_{0}(e^{B t_{B}(g,v)}v)) & \mbox{ if } v \neq 0\\
      (g^{-1}, 0) & \mbox{ if } v = 0 .
    \end{array}
    \right.
  \]
  which is easily verified that is the inverse map of $\psi$. It follows from the arguments above that $\psi$ is a homeomorphism that conjugate the flows $(e^{At}g, e^{At}v)$ and $(e^{Bt}g, e^{Bt}v)$ in $G \times \g$.

  As direct consequence we see that $\psi_1$ is a homeomorphism in $G$ that conjugate $ e^{At}$ and $e^{Bt}$ in $G$. \qed
\end{proof}

\section{Topological conjugate in $GL(n,\mathbb{R})$}

Let $A$ be a matrix in $gl(n, \mathbb{R})$. We wish to establish a condition to topological conjugate of the dynamical system
\[
  \dot{X} = A X, \,\, X \in Gl(n,\mathbb{R}).
\]
This system is well posted as the reader can be view in \cite{sachkov}. Following the idea of section above we consider the trivial bundle $GL(n,\mathbb{R}) \times gl(n,\mathbb{R})$ and we study the topological conjugacy of the flow generated by
\begin{equation*}
  \left(
  \begin{array}{c}
  \dot{X}\\
  \dot{V}
  \end{array}
  \right)
  =
    \left(
  \begin{array}{c}
  AX\\
  AV
  \end{array}
  \right),
\end{equation*}
that is, the flow given by $(e^{At}X,e^{At}V)$, where in both cases $e^{At} = \sum_{i=1}^{\infty} \frac{(tA)^n}{n!}$.  Our next step is to adopt a norm in $gl(n,\mathbb{R})$. Thus, taking a norm $| \cdot |$  on  $\mathbb{R}^n$ we adopt the supremun norm in $gl(n, \mathbb{R})$ 
\begin{equation}\label{euclidiannorm}
  \|A \| = sup\left\{ \frac{|A x|}{|x|}; 0 \neq x \in \mathbb{R}^n \right\}.
\end{equation}
The choice of the norm $\|\cdot \|$ allow us to show a well-known result between hyperbolic property and eigenvalues( see Theorem 5.1 in \cite[ch. 4]{robinson}). 

\begin{proposition}\label{inequalities}
  Let $A \in gl(n, \mathbb{R})$, and consider the equation $\dot{X} =  A X$. The following are equivalent.
  \begin{enumerate}
    \item There is a norm $\| \cdot \|_*$ and a constant $a>0$ such that for any initial condition $X \in GL(d,\mathbb{R})$, the solution satisfies
        \[
          \| e^{At} X \|_* \leq e^{-at} \|X\|_*, \ \ \textrm{for all} \ \ t > 0.
        \]
    \item for every generalized eigenvalue $\mu$ of $A$ we have $Re(\mu) <0$.
  \end{enumerate}
\end{proposition}
\begin{proof}
   We first observe that every generalized eigenvalue $\mu$ of $A$ has $Re(\mu)<0$ if and only if there is a norm $| \cdot |_*$ on $\mathbb{R}^n$ and a constant $a>0$ such that for any initial condition $x \in \mathbb{R}$, the solution of $\dot{x} = A x$ satisfies
   \[
     | e^A x |_* \leq e^{-at} |x|_*, \ \ \textrm{for all} \ \ t > 0.
   \]
  (see for example Theorem 5.1, chapter IV of Robinson). Thus it is sufficient to show that
  \[
    |e^{At}x |_* \leq e^{-at} |x|_* \Leftrightarrow \|e^{At}X \|_* \leq e^{-at} \|X\|_*.
  \]
  In fact, suppose that $|e^{At}x |_* \leq e^{-at} |x|_*$ such that $\dot{x} = A x$.
  Since $e^{At}X(x)$ is a solution for $\dot{x} = A x$, it follows that $|e^{At}X(x)|_* \leq e^{-at} |X(x)|_*$. We thus obtain
  \[
    \|e^{At}X \|_* \leq e^{-at} \|X \|_*.
  \]
  Conversely, suppose that $\|e^{At}X \|_* \leq e^{-at} \|X\|_*$, taking in definition (\ref{euclidiannorm}) $x = e_1$ it follows easily that
  \[
    |e^{At}x |_* \leq Ce^{-at} |x|_*
  \]
  with $x$ satisfying $\dot{x} = A x$.
  \qed
\end{proof}

Now we take a norm on $GL(n,\mathbb{R}) \times gl(n,\mathbb{R})$ as $\|(X, V) \| = \|V \|_*$. Thus, we are able to show a sufficiency condition to topological conjugacy in $GL(n,\mathbb{R})$.
\begin{theorem}\label{conjucacaotopologica}
  Let $A, B \in gl(n, \mathbb{R})$, and we consider dynamical systems  $\dot{X} = AX$ and $\dot{Y} = B Y$. If every generalized eigenvalues of $A$ and $B$ have real part negative, then	$e^{A t}$ and $e^{B t}$ are topologically conjugate in $GL(n,\mathbb{R})$.
\end{theorem}
\begin{proof}
  We first suppose that every generalized eigenvalues of $A$ and $B$ have real part negative. Thus Proposition \ref{inequalities} assures that there exist norms $\|\cdot \|_A$, 
	$\|\cdot \|_B$ and constants $a,b>0$ such that
  \begin{eqnarray*}
    \| e^{A t} X\|_A & \leq & e^{-at} \|X\|_A \\
    \| e^{B t} Y\|_B & \leq & e^{-bt} \|Y\|_B.
  \end{eqnarray*}
  in $gl(n,\mathbb{R})$. A direct application of Theorem (\ref{topologicallyconjugate}) with two distinct norm shows that then $e^{A t}$ and $e^{B t}$ are topologically conjugate in $GL(n,\mathbb{R})$. \qed
\end{proof}

\begin{corollary}
  Under hypothesis of Theorem \ref{conjucacaotopologica} with ``real part negative'' replaced by ``real part positive'', then 	$e^{A t}$ and $e^{B t}$ are topologically conjugate in $GL(n,\mathbb{R})$.
\end{corollary}
\begin{proof}
  It is sufficient to view that every generalized eigenvalues of $-A$ and $-B$ have real part negative and to apply Theorem above. \qed
\end{proof}

\begin{example}
  A matrix in $gl(2, \mathbb{R})$ can be written as   
	\[
	  \left(
		\begin{array}{cc}
		  a & b\\
			c & d
		\end{array}
		\right).
  \]
  Then eigenvalues of matrix above are given by $\lambda = \frac{(a+d) \pm \sqrt{(a-d)^2 + bc}}{2}$. It is direct that when 
	\begin{equation}\label{condition1}
	   (a-d)^2 + bc \leq 0
	\end{equation}
	the real part of eigenvalues depend only of trace $a+d$. Thus, suppose that two matrix $A$ and $B$ in $gl(2,\mathbb{R})$ satisfy the equation (\ref{condition1}),
	then 
	\begin{enumerate}
	  \item If $\rm tr(A) <0$ and $\rm tr(B)<0$, then $e^{At}$ and $e^{Bt}$ are topological conjugacy, by Theorem above.
		\item If $\rm tr(A) >0$ and $\rm tr(B)>0$, then $e^{At}$ and $e^{Bt}$ are topological conjugacy, by Corollary above. 
	\end{enumerate}
\end{example}


\section{Invariant flows on Semisimple  and Affine Lie Groups}

In this section, we study the topological conjugation on semisimple Lie group. Our idea is to transfer the dynamical system $\dot{g} = A_g$ on $G$ to one on the adjoint Lie group $Ad(G)$ and to find a condition to topological conjugation in this group. It is possible because we view $Ad(G)$ as a matrix group and we use the result of section above.

We begin by assuming that $G$ is a semisimple Lie group. We know that solution of $\dot{g} = A_g$ is given by 
\[
 g(t) = e^{At}g_0.
\]
Applying the adjoint operador in $g(t)$ we have
\[
  Ad(g(t)) = Ad(e^{At}g_0) = e^{ad(At)}Ad(g_0).
\]
Differentiate $Ad(g(t))$ with respect to $t$ we obtain
\begin{eqnarray*}
  Ad'(g(t)) & = & d(Ad)_{g(t)*} g'(t)  =  d(Ad)_{g(t)*} R_{g(t)*}(A)\\
	& = & R_{Ad(g(t))*}d(Ad)_{e}(A) =  R_{Ad(g(t))*}ad(A)\\
  & = & ad(A)\cdot Ad(g(t)).
\end{eqnarray*}

Thus taking two system $\dot{g} = A_g$ and $\dot{h} = B_h$ we have the dynamical systems on $Gl(\mathfrak{g})$
\begin{eqnarray*}
  Ad'(g(t)) & = & ad(A) Ad(g(t))\\
  Ad'(h(t)) & = & ad(B) Ad(h(t)).
\end{eqnarray*}
Suppose that every generalized eigenvalues of $ad(A)$ and $ad(B)$  have real part negative. From Theorem \ref{conjucacaotopologica} there exists a homeomorphism $\psi: Gl(\mathfrak{g}) \rightarrow Gl(\mathfrak{g})$  such that 
\[
  \psi(e^{ad(At)} Ad(g_0)) = e^{ad(Bt)} \psi(Ad(g_0)). 
\]  
Since $\mathfrak{g}$ is semisimple, it follows that $Ad$ is an isomorphism over its image (see for instance \cite{helgason}). We thus apply $Ad^{-1}$ at equality above to obtain
\begin{eqnarray*}
  Ad^{-1}(\psi(Ad(e^{At}) Ad(g_0))) & = & Ad^{-1}(Ad(e^{Bt}) \psi(Ad(g_0)))\\
  Ad^{-1}(\psi(Ad(e^{At}g_0))) & = & e^{Bt} Ad^{-1}\psi(Ad(g_0))). 
\end{eqnarray*}
Thus, if we denote $\varphi = Ad^{-1} \circ \psi \circ Ad$ we have 
\[
  \varphi(e^{At}g_0)  =  e^{Bt} \varphi(g_0).
\] 
Summarizing, 
\begin{theorem}\label{conjucacaotopologica}
  Let $A, B \in \g$, and we consider dynamical systems  $\dot{g} = A_g$ and $\dot{h} = B_h$. Suppose that $\g$ is semisimple Lie algebra. If every eigenvalues generalized of $ad(A)$ and $ad(B)$ have real part negative, then	$e^{At}$ and $e^{Bt}$ are topologically conjugate in $G$.
\end{theorem}

Our final step is to work with the class of  Affine groups. This idea follows from work due to Kawan, Rocio and Santana \cite{kawan}. Before we introduce the affine Lie groups and we study the topological conjugacy on it so we need to establish a result on semisimple Lie groups.

\begin{proposition}\label{hyperbolicdistance}
  Let $A \in \mathfrak{g}$ and suppose that $\mathfrak{g}$ semisimple Lie algebra. We consider dynamical system $\dot{g} = A_g$. Suppose that $G$ has a left invariant metric and denote by $\rho$ the distance associated to this metric. If $ad(A)$ has real part negative for every generalized eigenvalues, then there exist constants $a,c$ positives such that $\rho(e^{At}g_o, e) \leq c e^{-at} \rho(g_o,e)$. 
\end{proposition}
\begin{proof}
  Let $<,>$ be a left invariant metric on Lie group $G$. As $Ad$ is an isomorphism we adopt the following metric on $Ad(G)$: 
  \[
    <X, Y > = < d(Ad^{-1})X, d(Ad^{-1})Y>.
  \]
  Let us denote by $\rho_1$  the distance associated to metric on $Ad(G)$. It follows directly that $\rho_1(Ad(g(t)), Id )= \rho(g(t), e)$ for any smooth curve $g(t) \in G$. Suppose now that $Ad(c(t))$ satisfies the differential equation $Ad'(c(t)) = ad(A)\cdot Ad(c(t))$. Then Proposition \ref{inequalities} assures that there are positive constants $a, c$ such that 
  \begin{eqnarray*}
    \rho_1(Ad(g(t)), Id ) 
	  & = & \rho_1( e^{ad(A)t}Ad(g_0), Id) \leq  c^{-1} e^{-at} \rho_1(Ad(g_0), Id).
  \end{eqnarray*}
  We thus conclude that  $\rho(g(t), e) \leq c^{-1} e^{-at} \rho(g_0, e)$.
\end{proof}

Now we introduce the affine Lie groups. Let $V$ be an $n$-dimensional real vector space and $H$ a Lie group that acts on $V$. Take the group $G=H\rtimes V$ given by the semidirect product of $H$ and $V$. Recall that the affine group operation is defined by $(g,v) \cdot (h,w)=(gh,v+gw)$ for all $(g,v),(h,w) \in G$. Let $\pi:G\rightarrow H$ be the canonical projection of the affine group onto the Lie group $H$. The action of $G$ on $V$, given by $(g,v)\cdot w=gw+v$, with $(g,v)\in G$ and  $w\in V$, is called an affine action. The natural action of $\pi(G)=H$ on $V$ is called a linear action. Denote by $\g=\mathfrak{h} \rtimes V$ the Lie algebra of $G$, where $\mathfrak{h}$ is the Lie algebra of $H$. Note that for simplicity we use the same product symbol for the group and algebra. If $\Phi(t,g)=\exp_G(tX)g$ is a flow on $G$, where $X=(A,b) \in \g=\mathfrak{h} \rtimes V$ and $g\in G$, we denote by $\Phi^H$ the flow on $H$ given by $\Phi^H(t,h)=\exp_H(tA)h$ with $h\in H$.

\begin{proposition}
  Consider the affine group $G=H\rtimes V$ with $H$ being a semisimple Lie group with a left invariant metric. Let $\pi:G\rightarrow H$ be the canonical projection. Take flows $\Phi_i(t,g)=\exp(tX_i)g$ with $X_i=(A_i,b_i)$, $i=1,2$, on $G$, and suppose every eigenvalues generalized of $ad(A_1)$ and $ad(A_2)$ have real part negative. Then $\Phi_1$ is topologically conjugate to $\Phi_2$.%
\end{proposition}
\begin{proof}
  We first suppose that real part of all generalized eigenvalues of $ad(A_1)$ and $ad(A_2)$  are negative. From Proposition \ref{hyperbolicdistance} we have two unitary spheres  $\mathbb{S}^1_{A_1}$ and $\mathbb{S}^1_{A_2}$ in $G$ with center $g_0$ and $g_1$, respectively,   such that  $\exp(A_1t)g_0$ and $\exp(A_2t)g_1$ cross these in a unique time, respectively. It means that these unitary sphere are transversal sections to the flows $\exp(A_1t)g_0$ and $\exp(A_2t)g_1$, respectively.
	In this way, Proposition 2 in \cite{kawan} assures that $\Phi_1$ and $\Phi_2$ are topological conjugate.
\end{proof}

\end{document}